\documentclass[12pt]{amsart}
\usepackage[utf8]{inputenc}
\usepackage{slashed}
\usepackage{textcomp} 
\usepackage{amsmath,amstext,amsopn,amssymb, amsfonts, amsthm, mathtools}  
\usepackage{mathrsfs}
\usepackage{graphicx}
\usepackage{xcolor}
\usepackage{geometry} 
\usepackage{array} 
\usepackage{textcomp} 
\usepackage[all,cmtip]{xy} 
\usepackage{bbm}
\usepackage{varioref}
\usepackage{tikz-cd}

\usepackage{pdfsync}

\usepackage{enumitem}
\makeatletter
\def\namedlabel#1#2{\begingroup
    #2%
    \def\@currentlabel{#2}%
    \phantomsection\label{#1}\endgroup
}
\makeatother

\usepackage[pdftex,bookmarks,colorlinks,breaklinks]{hyperref}  
\definecolor{dullmagenta}{rgb}{0.4,0,0.4}   
\definecolor{darkblue}{rgb}{0,0,0.4}
\definecolor{darkgreen}{rgb}{0,0.4,0}
\hypersetup{linkcolor=darkblue,citecolor=blue,filecolor=dullmagenta,urlcolor=darkblue} 

\def\XXint#1#2#3{{\setbox0=\hbox{$#1{#2#3}{\int}$}
     \vcenter{\hbox{$#2#3$}}\kern-.5\wd0}}

\newtheorem{theorem}{Theorem}[section]
\newtheorem*{theorem*}{Theorem}
\newtheorem{lemma}[theorem]{Lemma}
\newtheorem*{lemma*}{Lemma}
\newtheorem{proposition}[theorem]{Proposition}
\newtheorem{corollary}[theorem]{Corollary}

\theoremstyle{definition}
\newtheorem{definition}[theorem]{Definition}
\newtheorem{example}[theorem]{Example}

\theoremstyle{remark}
\newtheorem{remark}[theorem]{Remark}
\newtheorem{question*}[theorem]{Question}

\numberwithin{equation}{section}

\newcommand{\grm}{\mathrm{germ}}
\newcommand{\sgn}{\mathrm{sgn}}

\newcommand{\pp}{{p(\cdot)}}
\newcommand{\pprime}{p^\prime(\cdot)}
\newcommand{\Lp}{L^{p(\cdot)}}
\newcommand{\Lpprime}{L^{p'(\cdot)}}
\newcommand{\Lpg}{L^{p(\cdot)}_{\mathrm{germ}}}
\newcommand{\lp}{\ell^{p(\cdot)}}
\newcommand{\lpprime}{\ell^{p'(\cdot)}}

\newcommand{\cpp}{{p'(\cdot)}}
\newcommand{\Pp}{\mathcal P}

\newcommand{\ba}{\mathrm{ba}}

\newcommand{\B}{\mathcal{B}}

\DeclareMathOperator*{\esssup}{ess\,sup}
\newcommand{\supp}{\mathrm{supp}}

\DeclareMathOperator{\acc}{acc}
\newcommand{\A}{\M}
\newcommand{\id}{\mathbf{1}} 

\newcommand{\linf}{\ell^{\infty}}

\newcommand{\M}{\mathcal{M}_\pp} 
\newcommand{\LpQ}{\Lpg} 
\DeclareMathOperator*{\essinf}{ess\,inf}

\newcommand{\R}{\mathbb R}
\newcommand{\N}{\mathbb N}

\newcommand{\Z}{\mathbb Z}

\title[On the dual of variable Lebesgue spaces]
{On the dual of variable Lebesgue spaces with unbounded exponent}

\author[A. Amenta]{Alex Amenta}
\address{\noindent Alex Amenta \newline
    \indent Mathematisches Institut
    \newline \indent Universit\"at Bonn, Bonn, Germany}
\email{amenta@math.uni-bonn.de}

\author[J.M.\ Conde-Alonso]{Jos\'e M. Conde-Alonso}
\address{\noindent Jos\'e M. Conde-Alonso \newline
  \indent Departamento de Matem\'aticas, Facultad de Ciencias
  \newline \indent Universidad Aut\'onoma de Madrid, Cantoblanco, Spain}
\email{jose.conde@uam.es}

\author[D. Cruz-Uribe]{David Cruz-Uribe, OFS}
\address{\noindent David Cruz-Uribe, OFS \newline
  \indent Department of Mathematics \newline \indent University of Alabama,
Tuscaloosa, AL 35487, USA}
\email{dcruzuribe@ua.edu}

\author[J. Oc\'ariz]{Jes\'us Oc\'ariz}
\address{\noindent Jes\'us Oc\'ariz \newline
  \indent Departamento de Matem\'aticas, Facultad de Ciencias
  \newline \indent Universidad Aut\'onoma de Madrid, Cantoblanco, Spain}
\email{jesus.ocariz@uam.es}

\thanks{The first author is supported by a Fellowship for Postdoctoral Researchers from the Alexander von Humboldt Foundation. Part of this research was completed while he was a postdoctoral researcher at the TU Delft, supported by the VIDI subsidy 639.032.427 of the Netherlands Organisation for Scientific Research (NWO). 
  The third author is supported by research funds from the Dean of the College of Arts \& Sciences, the University of Alabama. 
The fourth author is partially supported by the grant MTM2017-83496-P from the Spanish Ministry of Economy and Competitiveness and through the “Severo Ochoa Programme for Centres of Excellence in R\&D” (SEV-2015-0554)}

\makeindex

\begin{document}

\allowdisplaybreaks

\begin{abstract}
  We study the dual space of the variable Lebesgue space $\Lp$ with
  unbounded exponent function $\pp$ and provide an answer to a
  question posed in~\cite{fiorenza-cruzuribe2013}.  Our approach is to
  decompose the dual into a topological direct sum of Banach spaces.  The first
  component corresponds to the dual in the bounded exponent case, and
  the second is, intuitively, the dual of functions that live where
  the exponent is unbounded (in a heuristic sense).  The second space is extremely complicated, and
  we illustrate this with a series of examples.  In the special case
  of the variable sequence space $\ell^\pp$, we show that this piece
  can be further decomposed into two spaces,  one of which  can be
  characterized in terms of a generalization of finitely additive measures.
As part of our work, we also considered the question of  dense
 subsets in $\Lp$ for unbounded exponents.   We constructed two
 examples, one for general variable Lebesgue spaces $\Lp$ and one in the sequence space $\lp$.  This gives an answer
 to another question from~\cite{fiorenza-cruzuribe2013}.
\end{abstract}

\maketitle

\section{Introduction}
\label{section:introduction}

Variable Lebesgue spaces are a generalization of the classical
Lebesgue spaces $L^p$, in which the exponent $p \in [1,\infty]$ is
replaced by a function. Let 
$(\Omega, { \mathcal{A}}, \mu)$ be a measure space, and let $\Pp(\Omega)$ be the collection
of all measurable functions
$\pp \colon \Omega \to [1,\infty]$; we will refer to the elements of
$\Pp(\Omega)$ as exponent functions.   For each $\pp \in \Pp(\Omega)$, define the modular 
\[ 
\rho_{\pp}(f)
  :=
  \left(\int_{\Omega\setminus \Omega_\infty} \left|f(x)\right|^{p(x)} \; d\mu(x)
    +
    \|f \|_{L^\infty(\Omega_\infty)}\right), \]
where $\Omega_\infty := \{ x\in \Omega : p(x)=\infty \}$.  
Given a measurable function $f$, we say that $f\in \Lp{ (\Omega)}$ if
there exists $\lambda>0$ such that $\rho_{{\pp}}(f/\lambda)<\infty$.  This set
becomes a Banach function space when equipped with the Luxemburg norm
\begin{equation*}
  \|f\|_{L^\pp(\Omega)}
  :=
  \inf \left\{\lambda>0: \; \rho_{\pp}(f/\lambda)  {  \leq } 1 \right\}.
\end{equation*}
For ease of notation we usually omit mention of $\Omega$, writing
$\Lp$ for $\Lp(\Omega)$, $\|f\|_{\pp}$ for $\|f\|_{L^{\pp}(\Omega)}$, {$\mathcal{P}$ for $\mathcal{P}(\Omega)$} 
and so on.  When $\Omega=\N$, {$\mathcal{A}=\mathcal{B}(\N)$},  and $\mu$ is the discrete counting
measure, we will denote the space by $\ell^\pp$. 

The variable Lebesgue spaces were first introduced by
Orlicz~\cite{0003.25203}.  They have been {widely} studied for the
past thirty years, both for their intrinsic interest as function
spaces and for their applications to PDEs and the calculus of variations
with nonstandard growth conditions.  For further details of this
history, including extensive references, we refer the reader to the
monographs~\cite{fiorenza-cruzuribe2013,
 diening-harjulehto-hasto-ruzicka2010}.  The variable Lebesgue spaces
are an important special case of the more general Musielak--Orlicz
spaces~\cite{hasto-harjulehto2019,MR724434}.   

\medskip

In the study of the variable Lebesgue spaces, a fundamental distinction
is whether the exponent function is unbounded.  For convenience, we
define
\[ p_+:= \esssup_{x \in \Omega} p(x); \]
then it matters greatly whether $p_+=\infty$ or $p_+<\infty$.  In the
latter case, the spaces $\Lp$ behave much more like the classical
$L^p$ spaces, $1 \leq p<\infty$.  For instance, in this case bounded functions of
compact support and $C_c^\infty$ are {both} dense in $\Lp$.  Moreover, the
dual space can be completely characterized.   Define the dual exponent
$\cpp \in \Pp$ pointwise by
\begin{equation*}
  \frac{1}{p(x)} + \frac{1}{p'(x)} = 1,
\end{equation*}
where we use the convention that $1/\infty=0$.  Then if $p_+<\infty$,
the dual of $\Lp$ is isomorphic to $L^\cpp$, via the mapping $g \mapsto
\phi_g$, where $g\in L^\cpp$ and
\[ \phi_g (f)  = \int_\Omega f(x)g(x)\,d{\mu}(x).  \]
That $\phi_g \in (\Lp)^*$ follows at once from the generalized H\"older's
inequality,
\begin{equation}\label{eqn:holder}
  \int_\Omega |f(x)g(x)| \, d\mu(x) \leq C \|f\|_{\pp}
  \|g\|_{\pprime}. 
\end{equation}
The converse, that given $\phi \in (\Lp)^*$ there exists $g\in L^\cpp$
such that $\phi=\phi_g$, is more difficult, and the fact that $\pp$ is
bounded is central to the proof.  We note in passing that while {the
  map $g\mapsto \phi_g$ is an isomorphism} between $(\Lp)^*$ and
$L^\cpp$ it is not an isometry unless $\pp$ is constant.  This is
related to the fact that the sharp constant in~\eqref{eqn:holder} is
greater than $1$ unless $\pp$ is constant.
See~\cite[Theorem~2.80]{fiorenza-cruzuribe2013} for a detailed proof
of duality.

On the other hand, when $p_+=\infty$, the space $\Lp$ exhibits much
more pathological behavior.  Bounded functions of compact support are
not dense, the space is not separable, and it is possible to show that
there exist elements in $ (\Lp)^*$ that are not induced by elements of
$L^\cpp$.  Various authors have studied the topological and function
space properties of $\Lp$: see, for example,~\cite{MR3881841,
  MR1700499, MR3378865, MR3926175, MR2763369}.  However, to the best of
our knowledge no one has addressed the problem of characterizing the
dual of $\Lp$ when $p_+=\infty$.  This was stated explicitly as an
open problem in~\cite[Problem~A.3]{fiorenza-cruzuribe2013}, but prior
to this the question was
part of the folklore in the study of variable Lebesgue spaces.

The purpose of this paper is to provide new insight into the
dual space $(\Lp)^*$ when $p_+=\infty$.  We will always assume without
loss of generality that $p(x)<\infty$ $\mu$-almost everywhere.  For otherwise,
we have the natural decomposition 
\[ \Lp(\Omega) = \Lp(\Omega\setminus \Omega_\infty) \oplus  L^\infty(\Omega_\infty), \]
and so
\[ \Lp(\Omega) ^* \cong \Lp(\Omega\setminus \Omega_\infty)^* \oplus
  L^\infty(\Omega_\infty)^*.\]
The dual space $L^\infty(\Omega_\infty)^*$ is completely
characterized:  it is  isometrically isomorphic to the space
of finitely additive measures over $\Omega_\infty$ which are absolutely continuous with respect to $\mu$ (see Yosida and
Hewitt~\cite[Theorem~2.3]{yosida-hewitt1952}).  Therefore, it remains
to characterize the dual space when $\pp$ is everywhere finite but
unbounded.

The remainder of this paper is organized as follows.  In
Section~\ref{section1} we will consider the general case and give a
decomposition of $(\Lp)^*$ as the direct sum of $L^\cpp$ and the dual
of a quotient space we refer to as the germ space, and denote by
$L^\pp_\grm$.  Elements of this quotient space can be thought of as
generalized functions define{d} on the singularity of $\pp$, which contain the
essential data of how a function $f \in \Lp$ behaves with respect to
this singularity.  In referring to them as ``germs''  we borrow
terminology from sheaf theory.  We will give some basic properties of this space, and a
series of examples which illustrate the complicated nature of its dual
space.   We will also consider briefly the case when $L^\infty$ is
contained in $\Lp$, since in~\cite[Problem~A.3]{fiorenza-cruzuribe2013} it was
suggested that whether or not $L^\infty \subset \Lp$ might be pertinent to the problem.
Our examples suggest that in general it may not be.  

In Section~\ref{section2} we restrict
our attention to the special case of the variable sequence spaces
$\ell^\pp$.  In this case we show that the dual of the germ space can be further
decomposed into two pieces, one of which can be characterized in
terms of a natural generalization of finitely additive
measures.  In the special case when $\ell^\infty \subset \ell^\pp$,
one of these pieces is trivial, so at least in the setting of variable
sequence spaces this distinction matters.

One approach to  characterizing the elements of the dual space of $\Lp$ is to
find suitable dense subsets of $\Lp$ and describe the linear
functionals in terms of their action on these sets.   In Section~\ref{section:dense} we give a
description of a dense subset of $\Lp$ when $\pp$ is unbounded.  Again
in the special case of $\lp$ we give a different example which
actually lets us sketch a characterization of the third piece of the
dual space mentioned above.  This characterization is both complicated
and artificial, and suggests that further work is necessary to
fully describe the dual space.   We also note that our results on
dense subsets answers
another question raised in~\cite[Problem~A.2]{fiorenza-cruzuribe2013},
and we believe that these sets will be
useful in studying other problems on variable Lebesgue spaces with
unbounded exponents.

Throughout this paper our notation will be standard or defined as
needed.  For the variable Lebesgue spaces we will follow the notation
used in~\cite{fiorenza-cruzuribe2013}.  The value of the constants
$C$, $c$ may change from line to line, but they will always be
independent of the other quantities in the expression in which they
appear unless we specifically indicate this (for instance, by writing
$C_\pp$ to indicate the constant depends on $\pp$).  For positive
numbers $A$ and $B$, we write $A \lesssim B$ to mean that there exists
a constant $c > 0$ such that $A \leq cB$.  We write $A \simeq B$ to
mean that $A \lesssim B$ and $B \lesssim A$.


\section{General variable Lebesgue spaces} \label{section1}

Throughout  this section we fix a $\sigma$-finite measure space
$(\Omega,\mathcal{A},\mu)$.  Given an exponent function $\pp \in
\Pp(\Omega)$, 
we are most interested in the case where $p_+=\infty$  but we will
generally explicitly assume this fact if we need it.  However, following
the discussion in the Introduction, we will
always assume without comment that $\mu(\Omega_\infty)=0$.

\subsection{The dual of $\Lp(\Omega)$}
Given a
measurable set $E{\subseteq} \Omega$, let
\[ p_+(E) {:}= \esssup_{x\in E} p(x), \qquad  p_-(E) {:}= \essinf_{x\in E} p(x).  \]
We say that $E$ is
  {\it $\pp$-bounded} if $p_+(E) < \infty$.

    Define $\Lp_b \subset \Lp$ to be the subspace of functions with
    $\pp$-bounded support,
    \[ \Lp_b := \{ f \in \Lp :  p_+(\supp(f)) < \infty\},   \]
    and let $\overline{\Lp_b}$ be the closure
    of $\Lp_b$ in $\Lp$.  Let $\LpQ := \Lp / \overline{\Lp_b}$ denote
    the corresponding quotient space, with the usual quotient norm
  \begin{equation*}
    \|[f]\|_{\LpQ} := \inf\{\|f - g\|_{\Lp} : g \in \overline{\Lp_b}\}.
  \end{equation*}

  We can now state our main result.

  \begin{theorem}\label{dual-splitting}
Given an exponent function $\pp {\in} \Pp$ such that
$p_+=\infty$, the dual space $(\Lp)^*$ is isomorphic to {the external direct sum}
$\Lpprime \oplus (\Lp_\grm)^*$.
\end{theorem}

\begin{remark}
If $\pp$ is bounded, then every subset of $\Omega$ is $\pp$-bounded, so
$\Lp_b = \Lp$, and $\LpQ$ is trivial.  Thus
Theorem~\ref{dual-splitting} reduces to the known result when
$p_+<\infty$.    On the other hand, if $\pp$ is unbounded, then $\LpQ$
is always nontrivial; we will show this fact below in
Proposition~\ref{Lpb-not-dense}. 
\end{remark}

\medskip

To prove Theorem~\ref{dual-splitting},    recall
that given $g\in L^\cpp$, we can define $\phi_g \in (\Lp)^*$ by
\begin{equation} \label{eqn:prime-functional}
\phi_g(f) = \int_\Omega f(x)g(x)\,d\mu(x){.} 
\end{equation}
When $p_+=\infty$, this does not yield every element of the dual
space.    However, we do have that
every element of $(\Lp_b)^*$ is gotten in this way.  

\begin{proposition} \label{dualOfLpBounded} Given a functional
  $\phi \in (\Lp)^*$, there exists a unique function
  $g_\phi \in L^{\pprime}$ such that for all $f \in \overline{\Lp_b}$,
\begin{equation} \label{eq:dualityproduct}
  \phi(f) = \int_\Omega f(x) g_\phi(x) \, d\mu(x),
\end{equation}
and 
\begin{equation*}
  \|g_\phi\|_{p^\prime(\cdot)} \lesssim \|\phi\|_{(\Lp)^*}.
\end{equation*}
\end{proposition}

\begin{proof}
  Fix a functional $\phi \in (\Lp)^*$.  First note that if
  $f \in \Lp(\Omega)$ and $E \subset \Omega$, then $f|_E \in \Lp(E)$
  with $\|f|_E\|_{\Lp(E)} = \| \id_E f \|_{\Lp(\Omega)}$.  Conversely,
  if $h \in \Lp(E)$, then $\tilde{h}$, the extension by zero of $h$ to
  $\Omega$, is in $\Lp(\Omega)$ with
  $\| \tilde{h} \|_{\Lp(\Omega)} = \| h \|_{L^{p(\cdot)}(E)}$.  Given
  a $\pp$-bounded set $E \subset \Omega$, for all $h \in \Lp(E)$
  define
\begin{equation*}
	\phi_E(h) := \phi(\tilde{h}).
\end{equation*}
Then we have
\begin{equation*}
  |\phi_E(h)|
  =
  |\phi(\tilde{h})|
  \leq
  \| \phi \|_{\Lp(\Omega)^*} \|\tilde{h}\|_{\Lp(\Omega)}
  =
  \| \phi \|_{\Lp(\Omega)^*} \| h \|_{\Lp(E)},
\end{equation*}
so $\phi_E$ is in $\Lp(E)^*$ with
$\| \phi_E \|_{\Lp(E)^*} \leq \| \phi \|_{\Lp(\Omega)^*}$.   Since
$p_+(E)<\infty$, there exists a unique function
$g_\phi^E \in L^{\pprime}(E)$ such that
for all $f \in \Lp(E)$,
\begin{equation*}
	\phi_E(f) = \int_E f(x) g_\phi^E(x) \, d\mu(x),
\end{equation*}
and with
$\| g_\phi^E \|_{L^{p^\prime(\cdot)}(E)} \simeq \| \phi_E\|_{\Lp(E)^*}
\leq \|\phi\|_{(\Lp)^*}$.  

Since we can write $\Omega$ as the union of
an increasing sequence of $\pp$-bounded sets $\{E_i\}_{i=1}^\infty$,  by a standard patching argument there
exists a unique measurable function $g_\phi$ on $\Omega$ such that for
all $i$, ${g_\phi}|_{E_i} = g_\phi^{E_i}$ with
$\|\id_{E_i} g_\phi\|_{L^{\pprime}(\Omega)} \lesssim \| \phi
\|_{\Lp(\Omega)^*}$ with implicit constant independent of $i$.
Furthermore, for every $h \in \Lp_b$, we have
\begin{equation}\label{eqn:repn}
  \int_\Omega h(x) g_\phi(x) \, d\mu(x) = \phi(g).
\end{equation}

It remains to show that $g_\phi$ is in $\Lpprime(\Omega)$, and that \eqref{eqn:repn} holds for all $h$ in the closure $\overline{\Lp_b}$.
The functions $\id_{E_i} |g_\phi|$ increase pointwise a.e. to
$|g_\phi|$, so by the monotone convergence theorem for $\Lpprime$
\cite[Theorem 2.59]{fiorenza-cruzuribe2013} we have that
$g_\phi \in \Lpprime(\Omega)$ and
\begin{equation*}
  \|g_\phi\|_{\Lpprime(\Omega)} = \lim_{i \to \infty} \| \id_{E_i} g_\phi \|_{\Lpprime(\Omega)} \lesssim \| \phi \|_{\Lp(\Omega)^*}.
\end{equation*}
Finally, since
\begin{equation*}
  \int_\Omega |h(x) g_\phi| \, dx 
\lesssim 
\|h\|_{\Lp(\Omega)} \|g_\phi\|_{\Lpprime(\Omega)} 
< \infty
\end{equation*}
for all $h \in \Lp_b(\Omega)$, by the dominated convergence theorem
the representation \eqref{eqn:repn} holds for all $h$ in the closure
$\overline{\Lp_b}$.
\end{proof}

Our decomposition of the dual can now be proved using some basic
results from functional analysis.

\begin{proof}[Proof of Theorem~\ref{dual-splitting}]
We have the the exact sequence of bounded linear maps
\[
\begin{tikzcd}
  0 \arrow[r] & \Lp_b  \arrow[r, "i"] & \Lp  \arrow[r, "\pi"] & \LpQ  \arrow[r] & 0,
\end{tikzcd}
\]
where $i$ is the inclusion map and $\pi$ is the quotient map.  
This  induces the dual exact sequence
\[
\begin{tikzcd}
  0 \arrow[r] & (\LpQ)^*  \arrow[r, "i^*"] & (\Lp)^*  \arrow[r, "\pi^*"] & (\Lp_b)^*  \arrow[r] & 0.
\end{tikzcd}
\]
By ~\cite[Theorem~5.16]{MR1157815}, since $\pi^*$ is a bounded
 projection,  we can
write $(\Lp)^*$ as the (internal) direct sum
\[ (\Lp)^* = (\Lp_b)^* \oplus (\Lp_b)^\perp, \]
where $(\Lp_b)^\perp$ is the set of linear functionals that contain
$\Lp_b$ in their kernels.
%
To complete the proof, we first note that by \cite[Theorem
III.10.2]{MR1070713},  $(\Lp_b)^\perp$ is isometric{ally} isomorphic to $(\LpQ)^*$.
Finally, by  Proposition \ref{dualOfLpBounded}, we have that $(\Lp_b)^*\simeq \Lpprime$.
\end{proof}

\begin{remark}\label{rem:decompositionDual}
  In the first part of proof of Theorem~\ref{dual-splitting} we do not actually use
  any particular property of $\Lp$ as a Banach space.  Thus, we
  always have that for any Banach space $X$  and a closed subspace
  $Y$ we can write the dual as 
$$
X^*=Y^*\oplus (X/Y)^*,
$$ 
that is, $Y^*$ is a complemented subspace of $X^*$.
\end{remark}

\medskip

\subsection{The germ space $\Lp_\grm$}
To further characterize the dual of $\Lp$ we need to understand the
structure of $(\Lp_\grm)^*$, which in turn means understanding the
generalized functions in the germ space.  We give  several equivalent characterizations of the norm in
$\Lp_\grm$ and apply them to construct several examples.

\begin{lemma} \label{characlpgerm}
  Given $\pp \in \Pp$,  for every $f\in \Lp$,
  \begin{equation*}
    \Vert [f]\Vert_{\LpQ} = \inf \left\{\lambda >0:  \rho_{{\pp}}(f/\lambda) < \infty \right\}.
  \end{equation*}
\end{lemma}

\begin{proof}
  Fix $f\in \Lp$  and let  $\lambda>0$ be such that
  \begin{equation}\label{eqn:f-lambda}
    \int_\Omega \left| \frac{f(x)}{\lambda}\right|^{p(x)} \, d\mu(x) < \infty.
  \end{equation}
  Let $\{E_n\}_{n=1}^\infty$ be an increasing sequence of $\pp$-bounded
  subsets of $\Omega$ such that
  \[ \bigcup_{n \in \mathbb{N}} E_n = \Omega, \]
  and for $n\in \N$ define
  $f_n=f\id_{\Omega \setminus E_n} $, $g_n = f\id_{E_n}$.  Then $g_n
  \in \Lp_b$,   $f_n$
  converges pointwise a.e. to $0$, and  $[f_n] = [f]$ for all $n$.
  By the dominated convergence theorem we can find $n_0\in \N$ such
  that for all $ n\geq n_0$,
  \begin{equation*}
    \int_\Omega \left| \frac{f(x)-g_n(x)}{\lambda}\right|^{p(x)} \; d\mu(x)
    =
    \int_\Omega \left| \frac{f_n(x)}{\lambda}\right|^{p(x)} \; d\mu(x)
    =
    \int_{\Omega\setminus E_n} \left|
      \frac{f(x)}{\lambda}\right|^{p(x)} \;d\mu(x)
    \leq
    1.
  \end{equation*}
  Therefore, $\lambda \geq \| f- g_{n_0}\|_{\Lp}\geq \| [f] \|_{\Lp_\grm}$.

  \medskip
  
To prove the reverse inequality, 
let $g \in \Lp_b$ be arbitrary and let  $\lambda > 0$ be such that
  \begin{equation*}
    A := \int_\Omega \left| \frac{f(x) - g(x)}{\lambda} \right|^{p(x)} \, d\mu(x) < \infty.
  \end{equation*}
By the definition of the norm $\|g\|_{\pp}$, for all $\kappa>0$ we
have that
  \begin{multline*}
    \int_\Omega \left| \frac{g(x)}{\kappa} \right|^{p(x)} \, d\mu(x)
    = \int_\Omega \left| \frac{g(x)}{\|g\|_{\pp}} \frac{\|g\|_{\pp}}{\kappa} \right|^{p(x)} \, d\mu(x) \\
    \leq \max\{1,(\|g\|_{\pp}/\kappa)^{p_+(\mathrm{supp}(g))} \}< \infty.
  \end{multline*}
  Thus, we have that
  \begin{align*}
     \int_\Omega \left| \frac{f(x)}{\lambda} \right|^{p(x)} \,
     d\mu(x) 
    & 
      = \int_{\mathrm{supp}(g)} \left| \frac{f(x) - g(x) +
      g(x)}{\lambda} \right|^{p(x)} \, d\mu(x) \\
    & \qquad \qquad 
      + \int_{\Omega \setminus \mathrm{supp}(g)} \left| \frac{f(x) -
      g(x)}{\lambda} \right|^{p(x)} \, d\mu(x) \\ 
    & 
      \leq C_{p_+(\mathrm{supp}(g))}
      \int_{\mathrm{supp}(g)}\frac{|f(x) - g(x)|^{p(x)} +
      |g(x)|^{p(x)}}{\lambda^{p(x)}} \, d\mu(x) + A  \\
    &
      \leq C_{p_+(\mathrm{supp}(g))} \bigg( A + \int_\Omega
      \bigg|\frac{g(x)}{\lambda}\bigg|^{p(x)} \, d\mu(x) \bigg) + A \\
    &
      < \infty. 
  \end{align*}
  Therefore, we have the set inclusion
  \begin{multline*}
    \bigg\{ \lambda >0:  \int_{\Omega} \left|
      \frac{f(x)-g(x)}{\lambda}\right|^{p(x)} \; d\mu(x) < +\infty
    \bigg\} \\
    \subset \bigg\{ \lambda >0:  \int_{\Omega} \left| \frac{f(x)}{\lambda}\right|^{p(x)} \;d\mu(x) < +\infty \bigg\},
  \end{multline*}
  and it follows that 
  \begin{align*}
    \| [f] \|_{\LpQ} &= \inf_{g\in \Lp_b} \| f-g \|_{\Lp} \\
                              &\geq \inf_{g\in \Lp_b} \inf \left\{\lambda >0:  \int_{\Omega} \left| \frac{f(x)-g(x)}{\lambda}\right|^{p(x)} \; d\mu(x) < +\infty \right\} \\
                              &\geq \inf \left\{\lambda >0:  \int_{\Omega} \left| \frac{f(x)}{\lambda}\right|^{p(x)} \;d\mu(x) < +\infty \right\}.
\end{align*}
This completes the proof.
\end{proof}

As an application of Lemma~\ref{characlpgerm}  we show that if $\pp$ is
unbounded, then the germ space is never trivial.  

\begin{proposition}\label{Lpb-not-dense}
 Given $\pp\in \Pp$, if $p_+=\infty$, 
  then $\LpQ$ is nontrivial.
\end{proposition}

\begin{proof}
  For each $n \in \N$, define
  \begin{equation*}
    D_n := \{x \in \Omega : p(x) \in [2^n, 2^{n+1})\}.
  \end{equation*}
  Since $\pp$ is essentially unbounded, $D_n$ has positive measure for
  infinitely many $n$; let $\N'$ denote the set of all such $n$.  For
  each $n \in \N'$ fix a subset $E_n \subset D_n$ with
  $0<\mu(E_n) <\infty$.   (We do this  just to avoid the possibility that
  $\mu(D_n) = \infty$.)  Define a function $f$ by \begin{equation*}
    f(x) =  2\sum_{n \in \N'} \big( 2^{-n}
    \mu(E_n)^{-1}\big)^{1/p(x)}\chi_{E_n}.
  \end{equation*}
  Then $f\in \Lp$, since
  \begin{equation*}
    \int_\Omega \bigg| \frac{f(x)}{2} \bigg|^{p(x)} \, d\mu(x)
    =
    \int_\Omega \sum_{n \in \N'}  2^{-n}
    \mu(E_n)^{-1}\chi_{E_n}\,d\mu(x)
    =
    \sum_{n \in \N'}  2^{-n}
    < \infty.
  \end{equation*}
  On the other hand,
  
  \begin{equation*}
    \int_\Omega | f(x) |^{p(x)} \, d\mu(x)
    = 
    \sum_{n \in \N'} \frac{2^{-n}}{\mu(E_n)} \int_{E_n} 2^{p(x)} \,
    d\mu(x) \\
                       \geq \sum_{n \in \N'} 2^{-n} 2^n = \infty,
  \end{equation*}
  so by Lemma~\ref{characlpgerm}, $\|[f]\|_{\LpQ} \geq 1$.
Thus $\LpQ$ contains a nontrivial element.
\end{proof}

\begin{remark}
  In fact, if $\pp$ is essentially unbounded, then $\LpQ$ is infinite
  dimensional; we leave this as an exercise to the reader.
\end{remark}

\begin{remark}
  The following heuristic idea gives some intuition into how one can think of the space $\LpQ$.
  If the exponent $\pp$ is unbounded, then the way in which it diverges determines a kind of ``geometry'' of a hypothetical singular set.
  Elements of $\LpQ$ can be thought of as functions supported on this set.
  Of course, this set does not exist, and elements of $\LpQ$ are not functions.
  Nevertheless, this idea should be kept in the back of ones mind.
\end{remark}

The next two characterizations of the norm in the germ space reinforce
the intuition that $\Lp_\grm$ consists of generalized functions that
are, in some sense, supported where $\pp$ is infinite.
To prove them we first need to prove another
characterization of $\overline{\Lp_b}$.

\begin{lemma} \label{lemma:Lpb-vanish-infty}
  Given $\pp \in \Pp$ and $f\in \Lp$, $f \in \overline{\Lp_b}$ if and
  only if for any sequence of sets $\{E_k\}_{k=1}^\infty$ such
  that $E_k^c$ is $\pp$-bounded and $p_-(E_k)\rightarrow \infty$,
\begin{equation} \label{eqn:vanish1}
 \lim_{k\rightarrow \infty} \|f\id_{E_k}\|_\pp = 0. 
\end{equation}
\end{lemma}

\begin{proof}
Suppose first that \eqref{eqn:vanish1} holds.  Then for each $k$,
$f\id_{E_k^c} \in \Lp_b$, and so
$$\|f-f\id_{E_k^c}\|_\pp=\|f\id_{E_k}\|_\pp \rightarrow 0$$
as
$k\rightarrow \infty$.  Hence, $f\in \overline{\Lp_b}$.

Conversely, suppose $f\in \overline{\Lp_b}$.  Let
$\{h_j\}_{j=1}^\infty$ be a sequence in $\Lp_b$ that converges to
$f$.  Fix $\varepsilon>0$ and $j$ such that
$\|f-h_j\|_\pp<\varepsilon$.   Fix any sequence $\{E_k\}_{k=1}^\infty$
as in the hypotheses.  Then there exists $K$ such that if $k\geq K$,
$p_-(E_k) > p_+(\supp(h_j))$, so $\mu(\supp(h_j)\cap E_k)=0$.
Therefore,
\begin{multline*}
  \limsup_{k\rightarrow \infty} \|f\id_{E_k}\|_\pp
    \leq \limsup_{k\rightarrow \infty} \left( \|(f-h_j)\id_{E_k}\|_{\pp}
    + \|h_j\id_{E_k}\|_\pp\right) \\
    \leq \|f-h_j\|_{\pp} + \limsup_{k\rightarrow
      \infty}\|h_j\id_{E_k}\|_\pp
    = \|f-h_j\|_{\pp}  < \varepsilon. 
  \end{multline*}
  Since $\varepsilon>0$ is arbitrary, \eqref{eqn:vanish1} holds.
\end{proof}

\begin{lemma} \label{lemma:grm-norm-inf}
  {If} $\pp \in \Pp$,  then for all $f\in \Lp$,
  \[ \|[f]\|_{\Lp_\grm} = \inf\big\{ \|f\id_E\|_\pp :  E \subset
    \Omega, E^c \text{ is $\pp$-bounded} \big\}. \]
\end{lemma}

\begin{proof}
  Given a set $E\subset \Omega$ such that $E^c$ is $\pp$-bounded, we
  have that $f\id_E \in [f]$, and so by definition,
  $\|[f]\|_{\Lp_\grm}\leq \|f\id_E\|_\pp$.    To prove the reverse
  inequality, fix $\varepsilon>0$ and take $g\in [f]$.   Let
  $\{E_k\}_{k=1}^\infty$  be a sequence as in
  Lemma~\ref{lemma:Lpb-vanish-infty}.   Fix $\varepsilon>0$; then for
  all $k$ sufficiently large, $\|(g-f)\id_{E_k}\|_\pp <
  \varepsilon$.  Hence,
  \[ \|f\id_{E_k}\|_\pp \leq \|(f-g)\id_{E_k}\|_\pp + \|g
    \id_{E_k}\|_\pp
    < \varepsilon + \|g\|_\pp. \]
  Therefore, if we take the infimum over all sets $E$ with $E^c$
  bounded, we get
  \[ \inf_E \|f\id_E\|_\pp < \varepsilon + \|g\|_\pp.  \]
  If we now take the infimum over all such $g$, we get the desired inequality.
\end{proof}

In computing the germ norm we can replace the infimum by a limit over
a particular family of sets.

\begin{lemma} \label{lemma:grm-norm-limit}
 Given $\pp \in \Pp$, let $\{E_k\}_{k=1}^\infty$ be any sequence of
 sets such that $E_k^c$ is $\pp$-bounded, $E_{k+1}\subset E_k$, and
 $p_-(E_k) \rightarrow \infty$ as $k\rightarrow \infty$.  Then
 \[ \|[f]\|_{\Lp_\grm} = \lim_{k\rightarrow \infty}
   \|f\id_{E_k}\|_\pp. \]
\end{lemma}

\begin{remark}
  In practice we often apply Lemma~\ref{lemma:grm-norm-limit} with
  the sets $E_k := \{ x\in \Omega : p(x)\geq k \}$. 
\end{remark}

\begin{proof}
  Since the sets $E_k$ are nested, the sequence of norms is decreasing.
   Thus, the limit exists and in fact we
  have that
  \[ \inf_E \|f\id_E\|_ \pp \leq \inf_k \|f\id_{E_k}\|_ \pp
    = \lim_{k\rightarrow \infty} \|f\id_{E_k}\|_ \pp. \]

  To show the reverse inequality, fix a set $E$ such that $E^c$ is
  $\pp$-bounded.  Then for all $k$ sufficiently large, $p_-(E_k)>
  p_+(E^c)$, and so
  \begin{multline*}
    \lim_{k\rightarrow \infty} \|f\id_{E_k}\|_\pp
    \leq \lim_{k\rightarrow \infty} \big( \|f\id_{E_k\cap E}\|_\pp
    + \|f\id_{E_k\cap E^c}\|_\pp \big) \\
    \leq \lim_{k\rightarrow \infty} \|f\id_{E_k\cap E}\|_\pp
    \leq \|f\id_E\|_\pp. 
  \end{multline*}
  Since this is true for every such $E$, if we take the infimum we get
  the desired inequality.
\end{proof}

Finally, we can characterize the norm in a way analogous to the
associate norm on $L^\cpp$.

\begin{lemma} \label{lemma:grm-norm-associate}
  Given $\pp \in \Pp$,  we have for all $f\in \Lp$,
  \[ \|[f]\|_{\Lp_\grm} = \sup \bigg( \lim_{k\rightarrow \infty}
    \int_{E_k} f(x)h_k(x)\,d\mu(x) \bigg), \]
  where $E_k = \{ x\in \Omega : p(x)\geq k \}$, and the supremum is
  taken over all sequences $\{h_k\}_{k=1}^\infty$, such that
  $\supp(h_k)\subset E_k$ and $\|h_k\|_\cpp \leq 1$.  Moreover, there
  exists a sequence $\{h_k\}_{k=1}^\infty$ such that the supremum is
  attained. 
\end{lemma}

\begin{proof}
We use the precise constants for the associate norm
in $\Lp$~\cite[Theorem~2.34]{fiorenza-cruzuribe2013}:  for each $k>1$,
\[ \|f\id_{E_k}\|_\pp
    \leq 
\sup \bigg\{ \bigg|\int_{E_k} f(x)h_k(x)\,d\mu(x)\bigg| :
\|h_k\|_{\cpp}\leq 1 \bigg\} 
  \leq K_\pp(E_k) \|f\|_\pp, \]
where
\[ K_\pp(E_k) = \frac{1}{p_-(E_k)} + \frac{1}{p_+(E_k)}+1 \leq
    \frac{1}{k}+1. \]
The desired inequality follows from Lemma~\ref{lemma:grm-norm-limit}
if we take the limit as $k\rightarrow \infty$.   Furthermore, since we can
choose a sequence $\{h_k\}_{k=1}^\infty$ such that for each $k$
\[  \|f\id_{E_k}\|_\pp \leq \left(\frac{1}{k}+1\right)
    \bigg|\int_{E_k} f(x)h_k(x)\,d\mu(x)\bigg|, \]
  we can also find a sequence such that the supremum is attained.
 \end{proof}

Motivated by Lemma~\ref{lemma:grm-norm-associate} we can construct an
explicit collection of elements in $(\Lp_\grm)^*$.

\begin{proposition} \label{prop:dual-banach-limit}
  Given $\pp \in \Pp$ let $\{E_k\}_{k=1}^\infty$ be any sequence as in
  Lemma~\ref{lemma:grm-norm-limit}, and let $\{g_k\}_{k=1}^\infty$ be
  any sequence of functions $L^\cpp$ such that $\supp(g_k)\subset E_k$
  and
  \[ \limsup_{k\rightarrow\infty}\|g_k\|_\cpp = K > 0. \]
Then
  there exists $T\in (\Lp_\grm)^*$ such that $\|T\|_{ (\Lp_\grm)^*}
  \leq K$ and for all $f\in \Lp$,
  \begin{equation} \label{eqn:Banach1}
    T([f]) = \lim_{k\rightarrow \infty} \int_{E_k}
      f(x)g_k(x)\,d\mu(x), 
    \end{equation}
    whenever this limit exists.
  \end{proposition}

  \begin{proof}
    Let $\Phi\in \ell^\infty(\N)^*$ be a Banach limit.
    Intuitively,
    $\Phi$ is a 
    positive, bounded linear functional on $\ell^\infty$ which is gotten as an
    extension (via the Hahn-Banach theorem) from a linear functional
    on the subspace of convergent sequences. In particular, $\Phi$ is
    such that  for every $x=\{x_k\}_{k=1}^\infty\in \ell^\infty$, if
    $x_k\geq 0$ for all $k$, then $\Phi(x)\geq 0$, and if {$\{x_k\}_{k=1}^\infty$} is
    a convergent sequence, then
    \begin{equation} \label{eqn:Banach2}
       \Phi({\{x_k\}_{k=1}^\infty}) =\lim_{k\rightarrow \infty} x_k.  
     \end{equation}
    (See~\cite[II.4.21-22]{MR1009162}.)

    We now define
    \[ T([f]) := \Phi\bigg( \bigg\{ \int_{E_k}
      f(x)g_k(x)\,d\mu(x)\bigg\}_{k=1}^\infty \bigg).  \]
    The linear functional is well defined.  {By H\"older's inequality,
      the sequence inside $\Phi$ is bounded.  } Moreover,  if $h\in [f]$, then
    by Lemma~\ref{lemma:grm-norm-limit},
    \begin{multline*}
      \lim_{k\rightarrow \infty} \int_{E_k}
      |f(x)-h(x)||g_k(x)|\,d\mu(x) \\
      \leq \lim_{k\rightarrow \infty} \left(1+\frac{1}{p_-(E_k)}\right)\|(f-h)\id_{E_k}\|_\pp
      \|g_k\|_\cpp = 0. 
    \end{multline*}
    As a consequence, $T([f])=T([{h}])$.  The limit~\eqref{eqn:Banach1} follows at once from property~\eqref{eqn:Banach2}.
    Finally, essentially the same
    argument as above with $h=0$ shows that $\|T\|_{ (\Lp_\grm)^*}
    \leq K$.
  \end{proof}

  \begin{remark}
    The set of functionals $T = T(\Phi,\{g_k\}_{k=1}^\infty)$
    constructed in the proof of Proposition
    \ref{prop:dual-banach-limit} is norming for $\LpQ$: that is, for
   each $[f] \in \LpQ$, there exists  $T = T(\Phi;\{g_k\}_{k=1}^N)$
   with norm 1 such that
    \begin{equation*}
      T([f]) = \|[f]\|_{\LpQ}.
    \end{equation*}
    Thus, this set of functionals is relatively large in $(\LpQ)^*$.
    Fix $[f] \in \LpQ$; to find such a functional $T$, note that by
    Lemma \ref{lemma:grm-norm-associate} there exists a sequence
    $\{g_k\}_{k=1}^\infty$ such that $\supp(g_k) \subset E_k$, 
    $\|g_k\|_{\cpp} \leq 1$, and
    \begin{equation*}
      \lim_{k \to \infty} \int_{E_k} f(x) g_k(x) \, d\mu(x) = \|[f]\|_{\LpQ}.
    \end{equation*}
    Let $\Phi$ be any Banach limit, and define $T =
    T(\Phi, \{g_k\}_{k=1}^\infty)$; then  by~\eqref{eqn:Banach2} we have that
    \begin{equation*}
      T([f]) = \Phi\bigg( \bigg\{ \int_{E_k}
      f(x)g_k(x)\,d\mu(x)\bigg\}_{k=1}^\infty \bigg) = \|[f]\|_{\LpQ}.
    \end{equation*}
  \end{remark}

  \begin{remark}
    In Proposition~\ref{prop:dual-banach-limit}, the functional $T =
    T(\Phi, \{g_k\}_{k=1}^\infty)$ depends on the choice of the Banach limit $\Phi$, and since
    (by the Hahn-Banach theorem) there are an infinite number of such
    objects, $T$ is not unique.
    If we fix $\Phi$ and
    take a second {uniformly bounded} sequence $\{h_k\}_{k=1}^\infty$, it is unknown when
    this sequence will induce the same linear functional on
    $\Lp_\grm$.  If
\[ \lim_{k\rightarrow \infty}  a_k := \lim_{k\rightarrow \infty} \|(g_k-h_k)\id_{E_k}\|_\cpp = 0, \]
then they induce the same functional.   This follows from the same
    argument as at the end of the proof of
    Proposition~\ref{prop:dual-banach-limit}.  

    More generally, we have that this is also the case if
    $\{a_k\}_{k=1}^\infty \in \mathrm{Ker}(\Phi)$.  (If
    $a_k\rightarrow 0$ as $k\rightarrow \infty$, then $\{a_k\} \in
    \mathrm{Ker}(\Phi)$ but the converse need not be true.)
   To see this, let $T_1$ and $T_2$ be the functionals induced by
   $\{g_k\}_{k=1}^\infty$ and $\{h_k\}_{k=1}^\infty$, respectively.
   Then since $\Phi$ is a positive linear functional, and by H\"older's
   inequality in the variable Lebesgue spaces,
   \begin{multline*}
     |T_1(f)-T_2(f)| = \bigg| \Phi\bigg( \bigg\{ \int_{E_k} f(x)\big(
     g_k(x)-h_k(x)\big)\,d\mu(x) \bigg\}_{k=1}^\infty\bigg) \bigg|\\
     \leq \Phi\bigg( \bigg\{ \int_{E_k}\big| f(x)\big(
     g_k(x)-h_k(x)\big)\big|\,d\mu(x) \bigg\}_{k=1}^\infty\bigg)
     \leq \Phi(\{ C_ka_k \|f\id_{E_k}\|_{\Lp}\}_{k=1}^\infty).
   \end{multline*}
   We claim that the last term is equal to $0$.  First, clearly the
   sequence $\{a_k\}_{k=1}^\infty$ is bounded.  Second, as we noted
   above, $C_k\rightarrow 1$ since $p_-(E_k)\rightarrow \infty$.
   Therefore, by Lemma~\ref{lemma:grm-norm-limit}, 
   \[ \lim_{k\rightarrow \infty} a_k
       (C_k\|f\id_{E_k}\|_{\Lp}-\|[f]\|_{\Lp_\grm}) = 0.  \]
     Hence, again by the linearity of $\Phi$,
     \[ \Phi(\{ C_ka_k \|f\id_{E_k}\|_{\Lp}\}_{k=1}^\infty) =
       \|[f]\|_{\Lp_\grm} \Phi(\{a_k\}_{k=1}^\infty)  = 0.  \]

     While this condition is sufficient, we do not know if it is
     necessary, and the problem of characterizing these linear
     functionals remains open.
 \end{remark}
  
  \medskip

  We conclude this section with two examples that illustrate the
  pathological behavior of the germ space.   It was originally
  conjectured that $\Lp_\grm$ would behave like $L^\infty$,
  particularly in the special case when $L^\infty \subset \Lp$.
  Indeed, we initially hoped that in this case
  \[ \lim_{k\rightarrow \infty} \|f\id_{E_k}\|_{\pp}
    = \lim_{k\rightarrow \infty}
    \|f\id_{E_k}\|_{\infty}\|\id_{E_k}\|_\pp. \]
(The limit on the right-hand side exists because both sequences are
decreasing.)

  However, this is false:  the $\Lp_\grm$ norm of an $L^\infty$
  function can be arbitrarily smaller than its $L^\infty$ norm.  Let
  $\Omega=[1,\infty)$ and let $p(x)=\lfloor x \rfloor$.  For each $n>1$
  define the sets
  \[ F_n = \bigcup_{k=1}^\infty [k,k+n^{-k}].  \]
  Then $|F_n|<\infty$.  Furthermore, we have that
  \[ \int_1^\infty \big( n\id_{F_n})^{p(x)}\,dx
    =
    \sum_{k=1}^\infty n^k\cdot n^{-k} = \infty. \]
On the other hand, for any $\lambda>1/n$, 
  \[ \int_1^\infty \big( \lambda^{-1}\id_{F_n})^{p(x)}\,dx
    =
    \sum_{k=1}^\infty \lambda^{-k}\cdot n^{-k} < \infty, \]
  so if we truncate this integral to be on the interval $[k,\infty)$,
  it will be less than $1$ for all $k$ large.  Hence, by
  Lemma~\ref{lemma:grm-norm-limit},
  \[ \|\id_{F_n}\|_{\Lp_\grm}
    =
    \lim_{k\rightarrow \infty}\|\id_{F_n}\id_{[k,\infty)}\|_\pp =
    \frac{1}{n},\]
  but we clearly have
  \[ \lim_{k\rightarrow \infty}\|\id_{F_n}\id_{[k,\infty)}\|_\infty = 1. \]

  \medskip

  Moreover, we can modify this example to show that $\Lp_\grm$
  contains functions that are unbounded at infinity.   Define
  \[ f(x) = \sum_{k=1}^\infty k \id_{[k,k+k^{-k})}(x).  \]
  Then for all $\lambda>1$,
  \[ \int_1^\infty \bigg|\frac{f(x)}{\lambda}\bigg|^{p(x)}\,dx
    =
    \sum_{k=1}^\infty \frac{k^k}{\lambda^k}k^{-k} < \infty, \]
  but for all $\lambda \leq 1$ this integral diverges.  Hence $f\in
  \Lp(\Omega)$, and $\|f\|_{\Lp_\grm}\geq 1$.  

\section{Variable sequence spaces}\label{section2}

In this section we explore the structure of the germ
space and its dual in the case of the variable sequence spaces $\ell^\pp$, that is, when
$\Omega=\mathbb{N}$, $\mathcal{A}=\mathcal{B}(\N)$ and $\mu$ {is} the
counting measure.  In this case the fact that $\Omega$ is countable
and $\pp$ can be unbounded only at infinity greatly simplif{ies} the
situation.  But even in this special case we will see that the dual of
the germ space is still very complicated unless we impose additional
restrictions on the exponent function.   The sequence spaces have been
much less studied than their continuous counterparts; for the known
results, primarily for bounded exponents,
see~\cite{MR1899091,MR1931232,nekvindaP}. 

We fix some notation specific to this setting.  Given a sequence
$x=\{x(k)\}_{k=1}^\infty$, 
the modular $\rho_\pp(\cdot)$ and  norm  $\| \cdot \|_{\lp}$ are given
by 
\begin{equation*}
  \rho_{{\pp}}(f)= \sum_{k\in \mathbb{N}} \left| x(k)\right|^{p(k)}, \qquad
  \| x \|_{\lp}
  =
  \inf \bigg\{\lambda >0:
  \sum_{k\in \mathbb{N}} \left| \frac{x(k)}{\lambda}\right|^{p(k)} \leq 1 \bigg\}.
\end{equation*}

We consider three cases, which intuitively correspond to how
``far'' $\lp$ is from $\ell^\infty$.

\subsection{$\lp$ equals $\ell^\infty$}

We start with the observation that $\lp {\subseteq} \ell^\infty$,  as is
the case for the classical sequence spaces $\ell^p$.

\begin{lemma} \label{easyellpinclusion}
  Given any $\pp \in \Pp(\N)$, for all $x\in \lp$, 
\begin{equation*} 
  \| x \|_{\ell^\infty} \leq \| x \|_{\lp}.
\end{equation*}
\end{lemma}

\begin{proof}
  Fix $x \in \lp$ and  $\lambda>\|x\|_{{\lp}}$.  Then
\[     \sum_{k\in \N} \left| \frac{x(k)}{\lambda}\right|^{p(k)} \leq 1, \]
and so for every\ $k\in \N$,
 $|x(k)|^{{p(k)}} \leq \lambda^{p(k)}$.
  Therefore, $\| x \|_{\ell^\infty}\leq \lambda$, and taking the
  infimum over all such $\lambda$ we get the desired inequality.
\end{proof}

If we had the reverse inclusion, then we would have $\ell^\pp$
isomorphic to $\ell^\infty$, and we could use the classical
description of the dual of $\ell^\infty$.  This situation, however, is very
easy to characterize.

\begin{lemma} \label{characterizationellp}
  Given $\pp\in \Pp(\N)$,   the following are equivalent:
  \begin{enumerate}
  \item $\ell^\infty {\subseteq} \ell^\pp$ and   for every $f\in
    \ell^\infty$, $\| f \|_{\lp} \leq C_{p(\cdot)} \| f \|_{\ell^\infty}$.
  \item There exists $B>1$ such that
    \begin{equation} \label{eqn:char-ellp}
      \sum_{k\in \N} B^{-p(k)} < \infty. 
    \end{equation}

  \item $\id_{\N} \in \lp$. 
  \end{enumerate}
\end{lemma}

\begin{proof}
  To prove that  $(1)$ implies $(2)$, note that
$$
\| \id_{\N} \|_{\lp} \leq C_{p(\cdot)} \| \id_{\N} \|_{\ell^\infty} = C_{p(\cdot)},
$$
and so by the definition of the norm,
$$
\sum_{k\in\mathbb{N}} \frac{1}{ C_{p(\cdot)}^{p(k)}} \leq 1 < \infty.
$$

If $(2)$ holds, then by the dominated convergence theorem with respect
to counting measure, if we take $B$ sufficiently large we have that 
$$
\sum_{k \in \mathbb{N}} B^{-p(k)} \leq 1.
$$
Therefore, by the definition of the norm in $\ell^\pp$,
$\|\id_\N\|_\lp \leq B$.  Thus, $(3)$ holds. 

Finally, if $(3)$ is true, let
$C_{p(\cdot)} = \|\id_\N \|_{\ell^{p(\cdot)}}$. Then for any 
$x \in \ell^\infty$ and $k\in \N$, 
$|x(k)| \leq \|x\|_{\ell^\infty} \id_\N $.   Hence,  
$$
 \|x\|_{\ell^{p(\cdot)}} =  \| |x| \|_{\ell^{p(\cdot)}} \leq
 \|x\|_{\ell^\infty} \|\id_\N \|_{\ell^{p(\cdot)}} = C_{p(\cdot)}
 \|x\|_{\ell^\infty},
 $$
 and so $(1)$ holds.
\end{proof}

\begin{remark}
  When $r<\infty$, Nekvinda~\cite{MR1931232} characterized the exponents $\pp$ such
  that $\ell^\pp$ is isomorphic to $\ell^r$.
  Lemma~\ref{characterizationellp} extends his result to the
  case $r=\infty$. 
\end{remark}

\medskip

From the previous two lemmas we get the following characterization.

\begin{corollary}\label{p-inf-equivalence}
Given $\pp\in \Pp(\N)$, $\ell^\pp$ is isomorphic {to} $\ell^\infty$ if and
only if for some $B>1$, \eqref{eqn:char-ellp} holds.  
  Furthermore, in this case, $\ell^{\pprime}$ is isomorphic to  $\ell^1$.
  \end{corollary}

  \begin{proof}
    The isomorphism of $\ell^\pp$ and $\ell^\infty$ is immediate.  The
    second follows from the associate space characterization of the
    norm.  (For $\ell^1$ this is classical; for $\ell^\pp$
    see~\cite[Corollary~3.2.14]{diening-harjulehto-hasto-ruzicka2010}.)
    For all $f \in \ell^{\pprime}$ we have
    \begin{equation*}
      \|f\|_{\ell^{\pprime}}
      = \sup_{\substack{g \in \lp \\ \|g\|_{\ell^\pp} \neq 0}}
      \bigg|\sum_{n \in \N} f(n)g(n)\bigg|\|g\|_{\lp}^{-1}
      \simeq \sup_{\substack{g \in \ell^\infty \\ \|g\|_{\ell^\infty}
          \neq 0} }
      \bigg|\sum_{n \in \N} f(n)g(n)\bigg|\|g\|_{\ell^\infty}^{-1} = \|f\|_{\ell^1}.
    \end{equation*}
\end{proof}

In this special case, we can {easily} characterize the dual of $\ell^\pp$,
since it will be isomorphic to the dual of $\ell^\infty$, and
in particular, we can immediately identify the dual of the germ space
$\ell^\pp_\grm$.   The dual of $\ell^\infty$ is isomorphic to a space
of finitely additive measures~\cite{yosida-hewitt1952}.   More precisely,
\begin{equation*}
  \left(\linf \right)^* {\cong} \ba(\B(\N)),
\end{equation*}
where $\B(\N)$ is the $\Sigma$-algebra of subsets of $\N$ and
$\ba(\B(\N))$ is the set of finitely additive signed measures
$\mu$ on $\B(\N)$ with $|\mu|(\N) < \infty$.  The 
dual pairing can be identified with an integral:  
there exists an isomorphism
$\psi \colon (\linf)^* \to \ba(\B(\N))$ such that for all
$\phi \in (\linf)^*$ and $ x\in \linf$,
\begin{equation*}
  \phi(x) = \int_\N x \, d\psi(\phi).
\end{equation*}
Moreover,  $\ba(\B(\N))$ is isomorphic to the direct sum
$\ell^1 \oplus pba(\B(\N))$, 
where sequences in $\ell^1$ are identified with countably additive
measures on $\B(\N)$ and 
$pba(\B(\N))$ is the space of purely finitely additive measures
on $\B(\N)$.  (Recall that {a} measure $\mu$ is purely finitely additive if
$\mu(E)=0$ for all finite subsets $E\subset \N$.)  

If we combine these observations with
Theorem~\ref{eqn:prime-functional} and
Proposition~\ref{p-inf-equivalence}, we get the following
characterization of the dual of $\ell^\pp$.

\begin{theorem} \label{thm:dual-ellp-case1}
  Given $\pp \in \Pp(\N)$, suppose that for some $B>1$,
  \eqref{eqn:char-ellp} holds.  Then  $(\lp)^*$ is isomorphic to {the external direct sum}
\begin{equation*}
\ell^{\pprime} \oplus pba(\B(\N)).  
\end{equation*}
In particular, in this case we have that $(\lp_\grm)^*$ is isomorphic to
$pba(\B(\N))$.
\end{theorem}

\begin{remark} \label{remark:quick-growth}
The restriction that \eqref{eqn:char-ellp} holds is a very
strong one.  It is essentially a growth condition on $\pp$, and
requires $p(k)\rightarrow \infty$ quickly.  A simple example of this
is given by $p(k)=\log(k)^a$, $k\geq 2$ and $a>0$.  By the
integral test it is easy to see that~\eqref{eqn:char-ellp} holds if
and only if $a\geq 1$.  
\end{remark}

\medskip

\subsection{$\ell^\pp$ close to $\ell^\infty$}
To generalize Theorem~\ref{thm:dual-ellp-case1} we need to understand
better how far $\ell^\pp$ is from containing $\ell^\infty$.  We do
this by introducing the concept of a set with finite $\pp$-content.

\begin{definition}
Given a set $A\subseteq \mathbb{N}$, we say that it has \textit{finite
  $\pp$-content} if there exists a constant $B>1$ such that
$$
\sum_{k \in A} B^{-p(k)} < \infty.
$$
\end{definition}

If $p_+ < \infty$ then only finite subsets have finite $\pp$-content.
If $\ell^\infty\subset \ell^\pp$, then by
Lemma~\ref{characterizationellp}, $\N$ has finite $\pp$-content.  
However, given any unbounded exponent $\pp$, there exist infinite
subsets which have finite $\pp$-content:  no matter how slowly $p(k)$
grows, we can choose a set $A$ that is sufficiently sparse that it
will have finite $\pp$-content.  Denote by $\M$ the collection of all
subsets of natural numbers with finite $\pp$-content.

 \begin{remark}
   The set $\A$ is closed under intersections and finite unions, but
   is not closed under complements unless $\N$ has finite
   $\pp$-content.  If $\N$ does not have finite $\pp$-content, then
   the complement of {any} set which has finite $\pp$-content will
   not.  Thus, $\M$ {is} not in general a $\Sigma$-algebra.  It is
   however, a distributive lattice; it is {unbounded when $\N$ does not have
     finite $\pp$-content because} finite subsets of $\N$ always have
   finite $\pp$-content and it is impossible to find a maximal element
   of $\M$ with respect to inclusion.
 \end{remark}

 We now define a set function on $\M$, which, in some sense, measures
   how much a given set is affected by the singularity of $\pp$.
 
\begin{definition}
Given an exponent $\pp\in \Pp{(\N)}$, and a set  $A \in \M$, we define the
set function $w^\pp$ by
\begin{equation*}
  \omega^\pp(A):=\| [\id_A]\|_{\lp_\grm}.
\end{equation*}
When there is no possibility for confusion, we will write $\omega(A)$
instead of $\omega^\pp(A)$.  
\end{definition}

\begin{lemma}\label{lemw(AUB)}
Given an exponent $\pp\in \Pp{(\N)}$,  if $A,B\in \A$ have $\pp$-bounded
intersection, then, $\omega^\pp$ is  subadditive.  In fact,
\begin{equation*}
  \omega(A\cup B)= \max \left\{ \omega(A), \omega(B) \right\} \leq
  \omega(A)+\omega(B). 
\end{equation*}
\end{lemma}

\begin{proof}
  By Proposition \ref{characlpgerm},
  \begin{align*}
    \omega(A\cup B)
    &=\Vert [\id_{A\cup B}]\Vert_{\lp_\grm} \\
    & = \inf \bigg\{\lambda >0:
      \sum_{k\in A\cup B} \lambda^{-p(k)} < +\infty \bigg\} \\
    &=\max \bigg( \inf \bigg\{\lambda >0:
      \sum_{k\in A} \lambda^{-p(k)} < +\infty \bigg\},
      \inf \bigg\{\lambda >0:  \sum_{k\in B} \lambda^{-p(k)} < +\infty \bigg\} \bigg) \\
&=\max \left\{ \omega(A), \omega(B) \right\}.
  \end{align*}
\end{proof}

We can use the set function $\omega$ to compute the norms of certain
(simple) sequences in the germ space.

\begin{lemma}\label{lemSucConv}
Given an exponent $\pp\in \Pp{(\N)}$,  suppose  $x\in \lp$ is supported
in an infinite set $A\subset \N$, and furthermore that $x$ converges to $\alpha$ along all divergent sequences in $A$. Then
$$
\Vert [x] \Vert_{\lp_\grm} = |\alpha| \omega(A).
$$
\end{lemma}

\begin{proof}
  Fix $\varepsilon > 0$;
  then there exists $ N \in A$ such that $|x(k)-\alpha|<\varepsilon$
  for every $k\in A_N= \{ k \in A : k\geq N\}$.  But then
\begin{multline*}
  |\alpha - \varepsilon|\omega(A_N)
  =\Vert [(\alpha - \varepsilon)\id_{A_N}] \Vert_{\lp_\grm} \\
  \leq \Vert [x \id_{A_N}] \Vert_{\lp_\grm}
  \leq \Vert [(\alpha+ \varepsilon)\id_{A_N}] \Vert_{\lp_\grm}
  = |\alpha+ \varepsilon|\omega(A_N).
\end{multline*}
Since finite sets are $\pp$-bounded, for all $N$,
$\omega(A_N)=\omega(A)$ and $\|[x \id_{A_N}] \|_{\lp_\grm}= \|[x ]
\|_{\lp_\grm}$.   Therefore, since the above inequality holds for all
$\varepsilon>0$, we get the desired equality.
\end{proof}

We would like to extend this result to more general sequences.
Given $x\in \lp$, let $\acc(x)$ denote the set of limit
points of the sequence.  This set could be quite large--indeed, it
could contain an arbitrarily large interval.   We consider the special
case where $\acc(x)$ is finite.

\begin{proposition} \label{prop:normXwithmax}
Given $x\in \lp$, suppose  $\acc(x)=\{\alpha_i\}_{i=1}^n$.
Fix $\delta>0$ such that $|\alpha_i-\alpha_j|\geq \delta$, $i\neq j$,
and define $A_i=\left\{ k\in \N: |x(k)-\alpha_i|<\delta /2\right\}.$ Then
  \begin{equation}\label{eqn:grmnorm-acc}
    \| [x] \|_{\lp_\grm} = \max_{1\leq i \leq n} |\alpha_i| \omega(A_i){.}
  \end{equation}
\end{proposition}

\begin{proof}
Fix $0<\varepsilon<\delta/2$.  Then there exists $N>0$ such that if
$k\geq N$, there exists a unique $i$ such that $k\in A_i$ and
$|x(k)-\alpha_i|<\varepsilon$.  Therefore, arguing as we did in the
proof of Lemma~\ref{lemSucConv}, we have that
\[ (\alpha_i - \varepsilon)\omega(A_i) \leq
  \| [x\id_{A_i}]\|_{\lp_\grm}
  \leq \| [x]\|_{\lp_\grm}.  \]

Since this is true for all $i$,
\[  \| [x] \|_{\lp_\grm}  \geq \max_{1\leq i \leq n} |\alpha_i|
  \omega(A_i). \]

To prove the reverse inequality, fix $\lambda$ greater than the
right-hand side.  Since the sequence $x\id_{A_i}$ converges to
$\alpha_i$, By Lemmas~\ref{lemSucConv} and~\ref{characlpgerm},
\[ \sum_{k\in A_i} \bigg(\frac{|x(k)|}{\lambda}\bigg)^{-p(k)}< \infty.  \]
Since there are only a finite number of limit points, and since the
set $A_0 = \N\setminus \big( \cup_{i=1}^n A_i\big)$ is finite,
we have that
\[ \sum_{k=1}^\infty \bigg(\frac{|x(k)|}{\lambda}\bigg)^{-p(k)} = \sum_{i=0}^n
  \sum_{k\in A_i} \bigg(\frac{|x(k)|}{\lambda}\bigg)^{-p(k)}< \infty.  \]
Therefore, again by Lemma~\ref{characlpgerm}, $\lambda>\| [x]
\|_{\lp_\grm} $.  If we take the infimum over all such $\lambda$, we
get the desired inequality.
\end{proof}

Motivated by this example, we consider the following subspace of
$\lp$, which generalizes the simple functions.  Given $\pp \in
\Pp(\N)$, define the set of finite $\pp$-content simple functions to
be 
\[ S_\pp := \left\{x=\sum_{k=1}^N \alpha_k \id_{A_k} : A_k \in \M
    \text{ disjoint},
    \alpha_k \in \R \right \}. \]

In some sense, the set $S_\pp$ is larger than $\ell^\pp_b$, as the
next lemma shows.

\begin{lemma} \label{lemma:bounded-Spp}
  Given $\pp\in \Pp(\N)$, $\overline{\ell^\pp_b}\subset
  \overline{S_\pp}$.
\end{lemma}

\begin{proof}
Fix any $x\in \ell^\pp_b$  and let  $A=\supp(x)$.   By definition,
$p_+(A)<\infty$.  If $A$ is a
finite set, then $x\in S_\pp$, since all finite sets are in $\M$.  If
$A$ is infinite, then $x(k)\rightarrow 0$ as $k\rightarrow \infty$.
for if not, there exists $\epsilon>0$ and a subsequence
$\{k_j\}_{j=1}^\infty$ of $A$ such that $|x(k_j)|\geq \epsilon$, which
means that for any $\lambda>1$,
\[ \sum_{j=1}^\infty \bigg(\frac{|x(k_j)|}{\lambda}\bigg)^{p(k_j)} 
\geq (\epsilon\lambda^{-1})^{p_+(A)} \sum_{j=1}^\infty 1 = \infty, \]
which contradicts the fact that $x\in \lp$. 

But if $x(k)\rightarrow 0$, then any truncation of $x$ lies in $S_\pp$
and approximates $x$ in $\lp$ norm.  Hence, $x\in \overline{S_\pp}$,
and the desired inclusion follows at once.
\end{proof}

If $S_\pp$ is dense in $\lp$, we can characterize the dual of $\lp$,
generalizing Theorem~\ref{thm:dual-ellp-case1}.   To state our result,
we need the following definition, which generalizes the concept of
a finitely additive measure to the set $\M$ (which as we noted above is not
a $\Sigma$-algebra).

\begin{definition}\label{defPbaw}
  Given $\pp \in \Pp(\N)$, define $pba_\omega(\M)$ to be the vector
  space of set functions $\mu$ defined on $\M$
  that satisfy the following  properties:
  \begin{enumerate}
\item $\mu(A\cup B)=\mu(A)+\mu(B)$ for any pair of disjoint
  sets $A,\,B \in \M$.
  \item There exists $C>0$ such that given any collection
    $\{A_i\}_{i=1}^n$ of pairwise disjoint sets in $\M$, 
  $$
\sum_{i=1}^n \frac{\left| \mu(A_i) \right|}{w(A_i)}\leq C.
$$
\end{enumerate}
Define a  norm on $pba_\omega(\A)$  by 
$$
\| \mu \|_{pba_\omega}
:=
\inf \left\{C>0: \text{ condition (2) holds} \right\}.
$$
\end{definition}
\begin{remark}
  If we assume that $\omega$ takes the value 1 when the subset is infinite
  and 0 otherwise, then we recover the classical definition of $pba$ because
  condition (2) would imply that $\mu(A)=0$ for finite sets $A$
  and that $\mu$ has finite variation.
\end{remark}

  Since $\M$ is not a $\Sigma$-algebra, the elements of
  $pba_\omega(\M)$ are not finitely additive measures.  However, they
  are closely connected to the collection of finitely additive
  measures:  as the next example shows, we
  can construct elements of $pba_\omega(\M)$  from finitely additive
  measures defined on a fixed element of $\M$.

  \begin{proposition} \label{example:pbaw}
Given $\pp \in \Pp(\N)$, fix $D\in \M$.  Let $\mu \in pba(\B(D))$
and for $A\in \M$, define $\mu_D(A)= \mu(A\cap D)$.  Then $\mu_D\in
pba_w(\M)$.
\end{proposition}

\begin{proof}
  Property $(1)$ follows at once from the fact that $\mu$
  is a finitely additive measure.   
  Property  $(2)$ follows from
  Theorem~\ref{thm:dual-ellp-case1}.  Since 
  $pba(\B(D))$ is isomorphic to
  $\ell^\pp_\grm(D)^*$, let $\phi \in \ell^\pp_\grm(D)^*$ be the bounded
  linear functional associated to $\mu$.  Then we have that
  $$
  \sum_{i=1}^n \frac{\left| \mu_D(A_i) \right|}{w(A_i)}\leq
  \sum_{i=1}^n \frac{\left| \mu(A_i\cap D) \right|}{w(A_i\cap D)}
  =\phi\left(\sum_{i=1}^n \frac{\sgn( \mu(A_i\cap D))}{w(A_i\cap
      D)}\id_{A_i\cap D} \right)\leq \|\phi\|; 
  $$
the last inequality follows from the fact that by  Proposition \ref{prop:normXwithmax},
  $$
  \left\Vert \sum_{i=1}^n \frac{\sgn( \mu(A_i\cap D))}{w(A_i\cap
      D)}\id_{A_i\cap D} \right\Vert_{\lp_{\grm}}=1.
  $$
\end{proof}

\begin{theorem} \label{thm:dual-lp-case2}
Given $\pp \in \Pp(\N)$, suppose $S_\pp$ is dense in $\lp$.  Then
$(\lp)^*$ is isomorphic to {the external direct sum}
\[ \ell^\cpp \oplus pba_\omega(\M). \]
\end{theorem}

\begin{proof}
   By Theorem \ref{dual-splitting}, it suffices to prove that
   $(\lp_\grm)^*$ is isomorphic to $pba_\omega(\M)$.   In fact, we
   will show that there exists an isometric isomorphism.
   
 First,  given $\mu \in pba_\omega(\M)$ we will construct a linear
 functional $\phi_\mu$.   Given
 $x = \sum_{k=1}^N \alpha_k \mathbf{1}_{A_k}$ in $S_{\pp}$, we define
  \begin{equation*}
    \phi_\mu(x) := \sum_{k=1}^N \alpha_k \mu(A_k).
  \end{equation*}
  It is immediate that $\phi_\mu$ is linear.  Furthermore, by
  Propositon \ref{prop:normXwithmax} and the definition of
  $\|\mu\|_{pba_\omega(\M)}$ we have that
  \begin{multline*}
    |\phi_\mu(x)|
    = \bigg| \sum_{k=1}^N \alpha_k \omega(A_k) \frac{\mu(A_k)}{\omega(A_k)} \bigg| \\
    \leq \bigg( \max_{1\leq k \leq N} |\alpha_k| \omega(A_k) \bigg)
    \sum_{k=1}^N \frac{\mu(A_k)}{\omega(A_k)} 
  = \|[x]\|_{\lp_\grm} \|\mu\|_{pba_\omega}. 
\end{multline*}
Consequently, $\phi_\mu$ may be extended to all of $\lp_\grm$ by
density of $S_\pp$.
We then have $\phi_\mu \in (\lp_\grm)^*$, with
\[ \|\phi_\mu\|_{(\lp_\grm)^*} \leq \|\mu\|_{pba_\omega}. \]

Conversely,  fix $\phi \in (\lp_\grm)^*$.
  Define a set function $\mu_\phi \colon \M \to \mathbb{R}$ by setting
  \begin{equation*}
    \mu_\phi(A) := \phi([\id_A]);
  \end{equation*}
  since $A \in \M$, $\mathbf{1}_A$ is in $\lp_\grm$, and so $\mu_\phi$ is well-defined.
Property (1) in the definition of $pba_\omega$ follows
  immediately from the definition.   
    To prove property (2), suppose that $\{A_i\}_{i=1}^n$ is a finite
  collection of pairwise disjoint sets in $\M$.
  Then,   again  by Proposition \ref{prop:normXwithmax},
  \begin{multline*}
    \bigg| \sum_{i=1}^n \frac{\mu_\phi(A_i)}{\omega(A_i)} \bigg|
    = \bigg|\phi\bigg( \sum_{i=1}^n \frac{1}{\omega(A_i)} \id_{A_i} \bigg) \bigg| \\
    \leq \|\phi\|_{(\lp_\grm)^*} \bigg\| \sum_{i=1}^n
      \frac{1}{\omega(A_i)} \id_{A_i} \bigg\|_{\lp_\grm} 
    = \|\phi\|_{(\lp_\grm)^*}
      \max_{1\leq i \leq n} \omega(A_i)^{-1} \omega(A_i)
      = \|\phi\|_{(\lp_\grm)^*}.
  \end{multline*}
  Therefore,
  \begin{equation*}
    \|\mu_\phi\|_{pba_\omega} \leq \|\phi\|_{(\lp_\grm)^*}.
  \end{equation*}

  Finally, it is clear from the definitions that 
  \begin{equation*}
    \mu_{\phi_\mu} = \mu \qquad \text{and} \qquad \phi_{\mu_\phi} =
    \phi. 
  \end{equation*}
Hence, the mapping $\mu \mapsto \phi_\mu$ is an isometric
  isomorphism and our proof is complete.
\end{proof}

\begin{remark}
  The functional $\phi_\mu$ defined above can be thought of as a
  generalized integral with respect to $\mu \in pba_\omega(\M)$; we
  first define the integral on the dense set $S_\pp$ and then extend
  it in the usual way.  Since
  $\M$ is not a $\Sigma$-algebra, $\phi_\mu$ is not a classical
  integral.
\end{remark}

In light of Theorem~\ref{thm:dual-lp-case2} we would like to
characterize when $S_\pp$ is dense in $\ell^\pp$.  We have not been
able to do so.  Originally, we believed that it was dense if and
only if $\N$ has finite $\pp$ content, which would reduce
Theorem~\ref{thm:dual-lp-case2} to Theorem~\ref{thm:dual-ellp-case1}.
More precisely, we wanted to argue as follows:  if we modify the proof of
Proposition~\ref{prop:normXwithmax} we can show that for arbitrary
$x\in \ell^\pp$, 

$$
\| [x] \|_{\lp_\grm} \geq \sup_{\alpha\in \acc(x)} \left[|\alpha| \lim_{n\to\infty} w\left(\left\{k\in \N: |x(k)-\alpha|<1/n \right\} \right)\right]
$$
If the reverse inequality were true, then $S_\pp$ would always be dense in
$\ell^\pp$.  (We leave the details to the interested reader.)   This,
however, is not the case, as the next example shows.  It
remains an open question as to when $S_\pp$ is dense.

\begin{example}\label{exam:SpNotDense} There exists $\pp \in \Pp(\N)$ such that $S_\pp$ is
  not dense in $\lp$.
\end{example}

\begin{proof}
  Partition $\N$ as
  \[ \N = \bigcup_{s=1}^\infty A_s, \]
  where the sets $A_s$ are infinite and disjoint.  On each set $A_s$,
  define $\pp$ to be an increasing exponent such that for $k\in
  A_s$,  $p(k)=n$  $s^n$ times for each $n\geq s$.   We define the
  sequence $x$ by $x(k)=1/s$ if $k\in A_s$.   Then we have that $x\in
  \lp$; moreover, $\|[x]\|_{\lp_\grm}=1$.  To see this, fix any
  $\lambda>1$; then,
  \[ \sum_{k=1}^\infty \bigg(\frac{|x(k)|}{\lambda}\bigg)^{p(k)}
    = \sum_{s=1}^\infty \sum_{k \in
      A_s}\bigg(\frac{|x(k)|}{\lambda}\bigg)^{p(k)}
    = \sum_{s=1}^\infty \sum_{n= s}^\infty s^n
    \bigg(\frac{1/s}{\lambda}\bigg)^n =
    \frac{1}{(1-\lambda^{-1})^2\lambda} < \infty.  \]
  Hence, $x\in \lp$.  Similarly, this sum is infinite for any
  $\lambda\leq 1$, so by Lemma~\ref{characlpgerm},
  $\|[x]\|_{\lp_\grm}=1$.

  Finally, to show that $x\not \in \overline{S_\pp}$, note first that
  none of the sets $A_s$ have finite $\pp$-content:  arguing as
  before, for any $\lambda>0$,
  \[ \sum_{k \in
      A_s}\bigg(\frac{1}{\lambda}\bigg)^{p(k)}
    = \sum_{s=1}^\infty \sum_{n= s}^\infty s^n
    \bigg(\frac{1}{\lambda}\bigg)^n = \infty, \]
  so by Lemma~\ref{characlpgerm}, $\|[\id_{A_s}]\|_{\lp_\grm}=\infty$.
  Therefore, given any $y\in S_\pp$, there exists $s\in \N$ such that
  $y(k)=1/s$ for only a finite number of values of $k$.  Hence, there exists $N$
  large such that for any $\lambda>0$, 
  \[ \sum_{k=1}^\infty \bigg(\frac{|x(k)-y(k)|}{\lambda}\bigg)^{p(k)}
    = \sum_{\substack{k \in
        A_s \\k \geq N}}\bigg(\frac{|x(k)|}{\lambda}\bigg)^{p(k)}, \]
  and this sum is only finite for $\lambda>1$, and so $\|x-y\|_{\lp}\geq
  1$. 
\end{proof}

\begin{remark}
  By a careful choice of the sets $A_s$, we can take $\pp$ to be an
  increasing exponent.  In this case, for each $n\in \N$, $\pp$ takes
  on the value $n$ exactly $s^n$ times for $1\leq s \leq n$, and so
  \[ \sum_{s=1}^n s^n > n^n. \]
The first time $p(k)=n$ is for $k$ larger than
  \[ \sum_{j=1}^{n-1} \sum_{s=1}^j s^j < (n-1)\sum_{s=1}^{n-1} s^j <
    (n-1)^n.  \]
  Therefore, we have that $p(n^n)=n$.  For $x>e$,  if $\phi(x)=x\log(x)$,
  then $\phi^{-1}(x) \approx \frac{x}{\log(x)}$.  Hence, this shows
  that, roughly, 
  \[ p(k) \approx \frac{\log(k)}{\log\log(k)}.  \]
  Since if $p(k)=\log(k)$, we have that $\N$ has finite $\pp$-content,
  this example suggests that $S_\pp$ is dense in $\lp$ exactly when
  $\N$ has finite $\pp$-content.
\end{remark}

\medskip

Since understanding the closure of $S_\pp$ will be important for the
final characterization of the dual which we give below, we conclude
this section with a straightforward sufficient condition for when a
sequence lies in the closure.

\begin{proposition}\label{prop:w-dense}
  Let $x \in \lp$ and suppose that $\supp(x) \in \M$.
  Then $x \in \overline{S_\pp}$; furthermore,
  \begin{equation*}
    \|x\|_{\lp} \leq C \|x\|_{\ell^\infty},
  \end{equation*}
where the constant $C$ depends only on $\omega(\supp(x))$.  In
particular, if $\N$ has finite $\pp$-content, then $S_\pp$ is dense in
$\ell^\pp$. 
\end{proposition}

\begin{proof}
  Let $A := \supp(x)$.
  Since $x$ is supported on $A$, we have
  \begin{equation*}
    \|x\|_{\lp} = \|x\|_{\ell^{\pp|_A}(A)},
  \end{equation*}
  where $\pp|_A$ is the restriction of $\pp$ to $A$.
  Since $A$ has finite $\pp|_A$-content, by  Corollary
  \ref{p-inf-equivalence} we have that $\ell^{\pp|_A}(A)$ is
  isomorphic to $\ell^\infty(A)$.  Since simple functions are dense in
  $\ell^\infty(A)$ (see~\cite[IV.13.69]{MR1009162}), $x$ can be
  approximated by simple functions in $\ell^\infty(A)$, which are
  $\pp$-simple functions in $\ell^\pp(\N)$.  
\end{proof}

However, the converse of this result fails as the next example shows.
The growth condition allows for exponents $\pp$ that grow very slowly,
so it includes many exponents where $\N$ does not have fin{i}te
$\pp$-content.  Compare it to the examples in
Remark~\ref{remark:quick-growth}. 

\begin{proposition} \label{prop:in-Spp-closure}
  Given $\pp \in \Pp(\N)$, suppose $\N$ does not have finite
  $\pp$-content and 
  \[ \lim_{k\rightarrow 0} \frac{p(k)}{k} = 0.  \]
  Then there exists $x\in \ell^\pp$ such that $\supp(x)=\N$ (and so
  $\supp(x)\not\in \M$) but $x\in \overline{S_\pp}$.
\end{proposition}

\begin{proof}
  Define $x$ by $x(k)= 2^{-k/p(k)}$.  Then for any $\lambda>0$,
  \[ \lim_{k\rightarrow \infty}
    \left(\frac{x(k)}{\lambda}\right)^{p(k)/k} = \frac{1}{2}, \]
  and so by the root test,
  \[ \sum_{k=1}^\infty \left(\frac{x(k)}{\lambda}\right)^{p(k)} <
    \infty. \]
  Hence, by Lemma~\ref{characlpgerm}, $\|[x]\|_{\ell^\pp_\grm}=0$.
  Thus, $x \in \overline{\ell^\pp_b}$ and so by
  Lemma~\ref{lemma:bounded-Spp}, $x\in \overline{S_\pp}$. 
\end{proof}

\subsection{$\ell^\pp$ far from $\ell^\infty$}

In the case when $S_\pp$ is not dense in $\lp$ we are not able to
fully characterize the dual space of $(\lp)^*$.  We can, however, give a
direct sum decomposition which in some sense isolates the remaining
difficulties.

\begin{theorem}\label{thm:lp-dual-case3}
Given $\pp \in \Pp(\N)$, the dual space $\left(\ell^{p(\cdot)}\right)^*$ is
isomorphic to the external direct sum
$$
 \ell^{p'(\cdot)}\oplus pba_\omega(\A) \oplus \left(\lp /\overline{S_\pp}\right)^*.
$$
\end{theorem}

\begin{proof}
If we apply Remark \ref{rem:decompositionDual} to the closed subspace
$\overline{S_\pp}$, we can immediately decompose $(\lp)^*$ as the
internal direct sum
$$
(\lp)^*=(\overline{S_\pp})^* \oplus \left(\lp/\overline{S_\pp} \right)^*.
$$ 
By Theorem \ref{thm:dual-lp-case2} we have that 
$$
(\overline{S_\pp})^* \simeq \lpprime \oplus pba_\omega(\M). 
$$
This completes the proof.
\end{proof}

At the end of Section~\ref{section:dense} we use the dense subset we
define to sketch a possible characterization of the third term,
$\left(\lp/\overline{S_\pp} \right)^*$.  However, this
characterization is complicated and somewhat artificial, and the
problem of completely characterizing $(\lp)^*$ remains open.

\section{Dense subspaces of $\Lp$ and $\lp$}
\label{section:dense}

As the proofs of our main results in Sections~\ref{section1}
and~\ref{section2} show, there is a close connection between the
problem of characterizing $(\Lp)^*$ and finding practical dense
subsets of $\Lp$ when $\pp$ is unbounded.  This problem is of
independent interest and was raised as an open question
in~\cite[Problem~A.2]{fiorenza-cruzuribe2013} (though the connection
with dual spaces was not noted there).   In this section we give two
answers to this problem, one for general $\Lp$ spaces, and one
specifically for the sequence space $\lp$.   Neither has yielded a
satisfactory characterization of the dual space but we believe that
they are interesting in their own right and provide a foundation for
further work.

\subsection{A dense subset of $\Lp$}
We begin with a generalization of the simple functions to $\Lp$.
\begin{definition}
 Define the set of \emph{$\pp$-countable functions} in $\Lp$
        by
 $$
 \mathcal{C}_{\pp}:=\{f\in \Lp: f \text{ restricted to any $\pp$-bounded subset is a simple function}\} .
 $$
\end{definition}

\begin{remark}\label{rem:pp-countable}
  If $p_+ < \infty$, then a function is $\pp$-countable if and only if
  it is simple. In the particular case when $\Omega=\N$, every
  function is $\pp$-countable, so this set gives no new information in
  the case of variable sequence spaces.
\end{remark}

\begin{remark}
Any $\pp$-countable function takes at most a countable number of values.
\end{remark}




\begin{proposition}
Given $\pp\in \Pp$, $\mathcal{C}_{\pp}$ is dense in $L^\pp$.
\end{proposition}

\begin{proof}
  For each $k \in \N$ define $\Omega_k:=p^{-1}([k,k+1))$; then the
  $\Omega_k$ are measurable, $\pp$-bounded, pairwise disjoint, and
  their union is all of $\Omega$.  Fix $\varepsilon>0$.  The simple
  functions are dense in $\Lp(\Omega_k)$, so for every $k\in \N$
  there exists a simple function $g_k$, $\supp(g_k)\subset \Omega_k$,
  such that $\Vert f\id_{\Omega_k}-g_k\Vert_{\pp}< \varepsilon
  2^{-k}$.
  
  Define the  function $g$ to  coincide with $g_k$ in
  $\Omega_k$ for every $k\in \N$; then $g$  is $\pp$-countable.
  Moreover, by Minkowski's inequality,
  \[ \|f-g\|_{\Lp} = \bigg\| \sum_{k=1}^\infty
    (f\chi_{\Omega_k}-g_k)\bigg\|_{\Lp}
    \leq \sum_{k=1}^\infty \|f\chi_{\Omega_k}-g_k\|_{\Lp} <
    \varepsilon. \]
\end{proof}

\subsection{A dense subset of $\lp$}
As we noted above in Remark \ref{rem:pp-countable}, the set of
  $\pp$-countable functions in $\lp$ is the whole space, so we do not
  gain a strictly smaller dense subset. Moreover, as we showed in Example~\ref{exam:SpNotDense},
the collection $S_\pp$ may not be dense in $\lp$. However, we are
will show that we can build a {nontrivial} dense subset out of
sequences of elements in $S_\pp$.

\begin{definition} \label{defn:Z-space}
 Given $\pp \in \Pp(\N)$, for each $m
\in \Z$, define $S_\pp^m$ to be the subset of $S_\pp$ such that if
$x\in S_\pp^m$, then $|x(k)|\in (2^{m-1},2^m]\cup \{0\}$.
Define the set $\mathcal{Z}_\pp$ to consist of all
  sequences $x\in \lp$ such that  for some $M\in \Z$, there exists a
  sequence $\{y_m\}_{m=-\infty}^M$ such that $y_m\in S_\pp^m$,
  $\supp(y_m)\cap \supp(y_n)=\emptyset$ if $n\neq m$,  and 
  \[ x =  \sum_{m=-\infty}^M y_m.  \]
\end{definition}

\begin{remark}
  In Definition~\ref{defn:Z-space}, the sum of the $y_m$ makes sense
  pointwise, since for each $k$ there is at most one $y_m$
  which contains a non-zero element in the $k$-th coordinate, so the
  sum always converges.     However, this series may not converge in
  norm. 
\end{remark}

If $x\in \mathcal{Z}_\pp$, then  by Lemma~\ref{easyellpinclusion}, $\lp\subset \ell^\infty$,
and so there exists $M \in \Z$ such that 
$\acc(x)\subset [-2^M,2^M]$; moreover,  for each integer $m\leq M$,
$\acc(x)\cap \big( [-2^m,-2^{m-1})\cup (2^{m-1},2^m]\big)$ is finite.  
 Therefore, if $\acc(x)$ is
  infinite, it can be at most countable, and if this set has a limit
  point, it must be equal to $0$.

\begin{proposition}\label{prop:Zdense}
 Given $\pp \in \Pp(\N)$,  $\mathcal{Z}_\pp$ is dense in\ $\lp$.
\end{proposition}

\begin{proof}
  Since $x\in \ell^\infty$ by Lemma~\ref{easyellpinclusion}, there exists $M\in \Z$ such that  $0\leq x(k)\leq 2^M$.  For each $m$, $-\infty < m \leq M$, define
  \[ A_m := \{ k \in \N : |x(k)| \in (2^{m-1}, 2^m] \}.  \]
  Then, since for some $\lambda>0$,
  \[ \sum_{k\in A_m} \lambda^{-p(k)}
    \leq \sum_{k\in A_m} \bigg(\frac{x(k)}{2^m \lambda}\bigg)^{p(k)}
    \leq \sum_{k=1}^\infty  \bigg(\frac{x(k)}{2^m
      \lambda}\bigg)^{p(k)}
    <\infty, \]
  we have that $\omega(A_m)<\infty$.  Therefore, by
  Proposition~\ref{prop:w-dense}, $x_m= x\id_{A_m}$ is contained in the
  closure of $S_\pp$.  Moreover, since all the values of $x_m$ are
  either $0$ or contained in $(2^{m-1}, 2^m]$, $x_m$ is in the closure
  of $S_\pp^m$ and in fact we can approximate it by elements of
  $S_\pp^m$ whose supports are contained in $A_m$.  

  Fix $\varepsilon>0$ and fix $y_m \in S_\pp^m$ such that
  $\|x_m-y_m\|_{\lp} < 2^{m-M{-1}}\varepsilon$.   Define the sequence $y$
  {to be the one that coincides pointwise with the series}
  \[ y = \sum_{m=-\infty}^M y_m.  \]
 By the triangle inequality we have that
  $$
  \|y\|_{\lp}\leq \|x\|_{\lp}+\|x-y\|_{\lp}; 
  $$
  moreover, since the sets $A_m$ are disjoint, 
  \[ \|x-y\|_{\lp} \leq \sum_{m=-\infty}^M \|x_m-y_m\|_{\lp} <
    \varepsilon.  \]
Therefore, we have that $y\in \mathcal{Z}_\pp$ and, since $\varepsilon>0$ is
arbitrary, it follows that $\mathcal{Z}_\pp$ is dense in $\lp$.
\end{proof}

\medskip

We conclude this section by using the set $\mathcal{Z}_\pp$ to sketch a
characterization of  the third term of the dual
space $(\lp)^*$ given in Theorem~\ref{thm:lp-dual-case3}.  This
characterization is unsatisfactory since it is both complicated and
artificial, and we are not able to give reasonable examples of
elements in it.  Therefore, we present it as a starting point for
future work.

To describe this characterization, note first that each element of
$\mathcal{Z}_\pp$ can be identified with a sequence of ordered pairs
$\{(A_k,\alpha_k)\}_{k=1}^\infty$, where the $A_k$ are pairwise
disjoint sets in $\M$, and the $\alpha_k$ are real numbers such that
$\alpha_k\rightarrow 0$ as $k\rightarrow \infty$.  More precisely, for
each $m\leq M$, we have that
$y_m=\sum_{i=1}^{N_m}\alpha_{i,m} \id_{A_{i,m}}$ with $A_{i,m}$
disjoint.  We can then take the doubly indexed sequence
$\{(A_{i,m}, \alpha_{i,m} )\}$ and renumber the elements to get
$\{(A_k,\alpha_k)\}_{k=1}^\infty$. For ease of notation, we denote the associated
element of $\mathcal{Z}_\pp$ by $x(A_k,\alpha_k)$.

Conversely, given a sequence $\{(A_k,\alpha_k)\}_{k=1}^\infty$, where
the $A_k\in \M$ are pairwise disjoint and $\alpha_k\to 0$, we can
define a bounded sequence $\{x(j)\}_{j=1}^\infty$ by setting
$x(j)= \alpha_k$ if $j \in A_k$.  However, not every such sequence
induces an element of $\mathcal{Z}_\pp$: if $p(k)\rightarrow \infty$ slowly,
then we can choose the $A_k$ so that
$\|\id_{A_k}\|_{\lp} \rightarrow \infty$ while we choose the
$\alpha_k\rightarrow 0$ so slowly that $x(A_k,\alpha_k)$ is not in
$\lp$.  We will denote the larger collection of all such sequences by
$\mathcal{Z}(\M,c_0)$.

The goal is to follow the ideas in Theorem~\ref{thm:dual-lp-case2} and
define a collection of set functions on $\mathcal{Z}(\M,c_0)$ that mimic purely
finitely additive measures, but with the additional property that they
are zero on any element of $S_\pp$.   Denote these set functions by
$pba(\mathcal{Z}(\M,c_0) )$. Such set functions should
be linear:  
given $\{(A_k,\alpha_k)\}_{k=1}^\infty$, 
 $\{(B_k,\beta_k)\}_{k=1}^\infty$ in $ \mathcal{Z}(\M,c_0)$, 
\begin{enumerate}
\item if $A_k\cap B_k = \emptyset$ for all $k\in \N$, then
  \[ \rho(\{(A_k\cup B_k, \alpha_k)\}_{k=1}^\infty)
  = \rho(\{(A_k,\alpha_k)\}_{k=1}^\infty) + 
\rho(\{(B_k,\alpha_k)\}_{k=1}^\infty); \]
  
\item
  $\rho(\{(A_k,\alpha_k+\beta_k)\}_{k=1}^\infty)
  =\rho(\{(A_k,\alpha_k)\}_{k=1}^\infty)
  +\rho(\{(A_k,\beta_k)\}_{k=1}^\infty)$.
\end{enumerate}
They should also be bounded on $\lp{/\overline{S_\pp}}$: there exists a constant $C>0$
such that if $x(A_k,\alpha_k)\in \lp$, then
  \[ |\rho(\{(A_k,\alpha_k)\}_{k=1}^\infty)|
    \leq
    C \Vert [x\left(A_k,\alpha_k \right)]\Vert_{\lp{/\overline{S_\pp}}}. \]
 Given these properties,  we can mimic
  the proof of Theorem~\ref{thm:dual-lp-case2} to show that $(\lp)^*$
  is isomorphic to the external direct sum
\[   \ell^{p'(\cdot)}\oplus pba_\omega(\A) \oplus pba(\mathcal{Z}(\M,c_0)). \]
Details are left to the interested reader.

This characterization is impractical, as it is difficult to  exhibit
non-trivial elements of $pba(\mathcal{Z}(\M,c_0))$ even for very simple exponent
functions $\pp$.


 \bibliography{BibliographyVariableLp}{}
 \bibliographystyle{abbrv}

\end{document}